\documentclass[12pt]{article}
\usepackage{latexsym,color,amsmath,amsthm,amssymb,amscd,amsfonts}
\usepackage{tikz}
\usetikzlibrary{arrows}
\usepackage{scalefnt}

\usepackage{float}
\usepackage{graphicx} 

\usepackage{amsopn}
\usepackage{relsize}

\newtheoremstyle{nonum}{}{}{\itshape}{}{\bfseries}{.}{ }{\thmnote{#3}}

\newtheorem{thm}{Theorem}[section]
\newtheorem*{thm*}{Theorem}
\newtheorem*{prop*}{Proposition}

\newtheorem{cor}[thm]{Corollary}
\newtheorem{lem}[thm]{Lemma}

\newtheorem{rem}[thm]{Remark}

\newtheorem{exm}[thm]{Example}

\newtheorem{fact}[thm]{Fact}

\theoremstyle{nonum}

\newcommand{\R}{\mathbb R}

\def\K{{\cal K}}

\newcommand{\iprod}[2]{\langle #1,#2 \rangle} 

\newcommand{\E}{{\cal E}}

\def\K{{\cal K}}

\begin{document}
\title{Identifying set inclusion by projective positions and mixed volumes}
\date{}
\author{D.I. Florentin, V. D. Milman and A. Segal\\
		Tel Aviv University}
\maketitle
\begin{abstract}
We study a few approaches to identify inclusion (up to a shift) between
two convex bodies in $\R^n$. To this goal we use mixed volumes and
fractional linear maps. We prove that inclusion may be identified
by comparing volume or surface area of all projective positions of
the sets. We prove similar results for Minkowski sums of the sets.
\end{abstract}

\section{Introduction and Results}
Set inclusion $A\subseteq B$ of two convex bodies (elements of $\K^n$,
namely compact convex non degenerate sets), implies that for every
monotone functional $f:\K^n\to\R$, one has by definition, $f(A)\le f(B)$.
For example, for the volume functional we have $|A|\le|B|$. Our goal is
to achieve a reverse implication: describing a family $\cal{F}$ of such
functionals, with the property of {\em identifying inclusion}, that is,
given $A,B\in \K^n$, if $f(A)\le f(B)$ for all $f\in{\cal F}$, then
$A\subseteq B$ (or more generally, $B$ contains a translate of $A$). Note
that if a family of functionals $\cal{F}$ identifies inclusion, then it
{\em separates elements} in $\K^n$, that is, if $f(A)=f(B)$ for all $f\in
\cal{F}$, then $A=B$ (or more generally, $B$ is a translate of $A$). The
converse, however, is not true in general. That is, some families
separate points but do not identify inclusion. For example, see the
theorem by Chakerian and Lutwak below.

R. Schneider showed in \cite{Schneider_Mixed} that a convex body is
determined, up to translation, by the value of its mixed volumes with
some relatively small family of convex bodies. An extension to this
fact was given by W. Weil in the same year:
\begin{thm*}[Weil,\cite{Weil}] Let $A,B \in \K^n$. Then $B$ contains
a translate of $A$ if and only if for all $K_2, \ldots, K_n \in \K^n$:
$$ V(A, K_2, \ldots,  K_n) \le V(B, K_2, \ldots, K_n),$$  where
$V(K_1, \dots, K_n)$ denotes the $n$-dimensional mixed volume.
\end{thm*}
Actually, it is possible to reduce even further the information on the bodies
$A$ and $B$, as follows from the result of Lutwak \cite{Lutwak} which we discuss in
Section \ref{Sec_Sums}.

We will investigate a few approaches to achieve the same goal. First we
use a family of transformations on $\R^n$. We examine two such families
in this note: the group of affine transformations $AF_n$, and the group
of the far less explored fractional linear (or projective) transformations
$FL_n$. For example we may consider ${\cal F}_L=\{f\circ T:T\in AF_n \}$,
where $f$ is the volume functional, and $T$ is considered as a map on
$\K^n$. Unfortunately, the affine structure respects volume too well,
that is, if $|A|\le |B|$, then for every $T\in AF_n$ we have $|TA|\le
|TB|$ as well. In other words, the action of the affine group is not rich
enough to describe inclusion. However, it turns out that the larger family
${\cal F}_P=\{f\circ T:T\in FL_n \}$ is identifying inclusion. We will
consider replacing the volume functional by a mixed volume, as well as
replacing the family of transformations by different operations (such as
Minkowski sums with arbitrary bodies). We would like to emphasize that
the proofs we present in this note are not very sophisticated. However,
they point to some directions which Convexity Theory did not explore
enough, and lead to new and intriguing questions.

Let us introduce a few standard notations. First, for convenience we will
fix an Euclidean structure and some orthonormal basis $\{e_i\}_1^n$. For
a vector $x \in \R^n$ we will often write $x = (x_1, \ldots, x_n) =
\sum_1^n x_i e_i$. Given a subspace $E \subset \R^n$, $P_E$ will denote
the orthogonal projection onto $E$. We denote by $D_n$ the $n$-dimensional
Euclidean ball and, for $1\leq i \leq n$, by $W_{n-i}(K)$ the quermassintegral:$$
W_{n-i}(K)=V(K[i], D_n[n-i]).$$
For further definitions and well known properties
of mixed volumes and quermassintegrals see \cite{Schneider_Book}.
We denote the support function of $K$ by:
$$ h_K(u) = \sup_{x\in K}{\iprod{x}{u}}. $$
Let us recall the definition 
of fractional linear maps. We identify $\R^n$
with a subset of the projective space $RP^n$, by fixing some point
$z\in \R^{n+1}\setminus\{0\}$, and considering the affine subspace
$E_n=\{x|\iprod{z}{x}=1\} \subset\R^{n+1}$. Every point in $E_n $
corresponds to a unique line in $R^{n+1}$ passing through the origin. A
regular linear transformation $\tilde{L}:\R^{n+1}\to\R^{n+1}$ induces an
injective map $L$ on $RP^n$. A fractional linear map is the restriction
of such a map to $E_n\cap L^{-1}(E_n)$. The maximal (open) domain
$Dom(F)$, of a non-affine fractional linear map $F$ is $\R^n\setminus H$,
for some affine hyperplane $H$. Since our interest is in convex sets,
we usually consider just one side of $H$ as the domain, i.e. our maps are
defined on half spaces. They are the homomorphisms of convexity, in the
sense that there are no other injective maps that preserve convexity of
every set in their domain. The big difference, compared to linear maps,
is that the Jacobian matrix is not constant, and its determinant is not
bounded (on the maximal domain). In Section \ref{Sec_FL+Quermass} we show
that $\mathcal{F}_P$ is an identifying family:

\begin{thm}\label{Thm_Vol-FL}
Let $n\ge1$, and $A,B\in\K^n$. If for every admissible $F\in FL_n$,
one has $$Vol(FA)\le Vol(FB),$$ then $A\subseteq B$.
\end{thm} When we say that $F$ is admissible we mean that $A,B\subset Dom(F)$.
Theorem \ref{Thm_Vol-FL} is a particular case of the following fact,
where the volume is replaced by any of the quermassintegrals:
\begin{thm}\label{Thm_iVol-FL}
Let $n\geq 2$, $A,B\in\K^n$, and fix $1\le i\le n$. If for every admissible
$F\in FL_n$ one has $$W_{n-i}(FA)\le W_{n-i}(FB),$$ then $A\subseteq B$.
\end{thm}

The proof of Theorem \ref{Thm_iVol-FL} is based on the non boundedness
of the Jacobians of admissible fractional linear maps. That is, if
$A \setminus B$ is of positive volume, we may choose an admissible
$F\in FL_n$ such that $FA$ exceeds $FB$ in volume or, say, surface area,
regardless of how small $|A \setminus B|$ is. The exact formulation is
given in Lemma \ref{Lem_FL-Diverg}.

Had we considered only $F\in FL_n$ which are affine in Theorem \ref{Thm_Vol-FL},
clearly the conclusion could not have been reproduced (since $|FA|=
\det(F)|A|$, we in fact only assume that $|A|\le |B|$). One may
ask the same question in the case of surface area, namely $i=n-1$. Since
the surface area is not a linear invariant, not even up to the determinant
(as in the case of volume), the answer is not trivial, but it does follow
immediately (along with the restriction to $n\ge3$) from the negative
answer to Shephard's problem. In \cite{Petty,Schneider_Shep} Petty and
Schneider showed:
\begin{thm*}[Petty, Schneider]
Let $n\ge3$. Then there exist centrally symmetric bodies $A,B\in\K^n$,
such that $Vol(A)>Vol(B)$, and yet for every $n-1$ dimensional space $E$,
$Vol_{n-1}(P_E A)\le Vol_{n-1}(P_E B)$. In particular, $A\not\subseteq B$.
\end{thm*}

We say that $K\in\K^n$ is centrally symmetric (or symmetric), if $K=-K$
(i.e. its center is $0$). Chakerian and Lutwak \cite{Chak-Lut} showed that
the bodies from the previous theorem satisfy a surface area inequality in
every position. For the sake of completeness, we append the proof.
\begin{thm*}[Chakerian, Lutwak]
Let $n\ge3$. Then there exist centrally symmetric bodies $A,B\in\K^n$,
such that $A\not\subseteq B$, and yet for every $L\in AF_n$ we have
$$Vol_{n-1}(\partial LA)\le Vol_{n-1}(\partial LB).$$
\end{thm*}
\begin{proof} Let $A, B$ be the sets whose existence is assured by the
previous theorem. Recall the definition of the projection body $\Pi K$
of a convex body $K$: 
$$ h_{\Pi K}(u) = Vol_{n-1}(P_{u^\perp}(K)).$$
Using this notion, the assumption on $A$ and $B$ can be reformulated as
$\Pi A \subseteq \Pi B$. By Kubota's formula $W_1(\Pi K)=c_n|\partial K|$,
which implies that $|\partial A| \le |\partial B|$. Let $L \in AF_n$.
There exists $\tilde{L} \in GL_n$ such that for all $K\in\K^n$, $\Pi LK =
\tilde{L}\Pi K$. Thus, the inclusion $\Pi LA \subseteq \Pi LB$ holds as
before, and we get for all $L \in AF_n$: $$|\partial LA| \le |\partial LB|.$$
\end{proof}

In Section \ref{Sec_FL+Quermass} we recall some of the properties of
fractional linear maps and gather some technical lemmas, to deduce
Theorem \ref{Thm_iVol-FL}.

In Section \ref{Sec_Sums} we consider a different type of an indentifying
family of functionals. The volume of projective (or linear) transformations
of a body $A$ is replaced by the volume of Minkowski sums of $A$ with
arbitrary bodies. First let us mention the following curious yet easy
facts, which, in the same time, demonstrate well our intention.
\begin{thm}\label{Thm_Sum-Sym}
Let $n\ge1$, and let $A,B\in\K^n$ be centrally symmetric bodies. If for
all $K\in\K^n$, one has $|A+K|\le |B+K|$, then $A\subseteq B$.

Moreover: \\
\noindent Let $n\geq 1$, $K_0\in\K^n$, and let $A, B \in \K^n$ be
centrally symmetric bodies. If for every linear image $K=LK_0$, one has $|A+K|\le
|B+K|$, then $A \subseteq B$.
\end{thm}

\begin{thm}\label{Thm_Sum-Gen}
Let $n\ge1$, and let $A,B\in\K^n$. If for every $K\in\K^n$, one has
$|A+K|\le |B+K|$, then there exists $x_0\in \R^n$ such that $A\subseteq B+x_0$.

Moreover:\\
Let $n \geq 1$, and let $A, B \in \K^n$. If for every
simplex $\Delta\in\K^n$, one has $|A + \Delta| \leq |B+ \Delta|$, then there
exists $x_0\in \R^n$ such that $A \subseteq B + x_0$.
\end{thm}

We include a direct proof of the symmetric case, since in this case the
argument is far simpler. To prove the general case, we first obtain the
inequality:
$$\forall K\in\K^n:\qquad V(A,K[n-1])\le V(B,K[n-1]),$$ where $V$ is the
$n$ dimensional mixed volume, and $K[n-1]$ stands for $n-1$ copies of the
body $K$. Finally, the proof may be completed by applying a beautiful result
of Lutwak to the last inequality (see Section \ref{Sec_Sums}). In the rest
of Section \ref{Sec_Sums} we investigate the situation where we only assume
that the $n-1$ dimensional volume of every section of $A+K$ is smaller than
that of $B+K$ (for every $K$).

In the last two sections, we formulate conditions to achieve inclusions
between two n-tuples of convex bodies, $K_1,\dots,K_n,\,L_1,\dots,L_n$, in
terms of the mixed volume of their affine or projective positions.

\section{Fractional linear maps}\label{Sec_FL+Quermass}
Let us recall the definition and some properties of fractional linear maps.
A fractional linear map is a map $F:\R^n\setminus H\to \R^n$, of the form
$$F(x)=\frac{Ax+b}{\iprod{x}{c}+d},$$ where $A$ is a linear operator in
$\R^n$,  $b, c \in \R^n$ and $d \in \R$, such that
the matrix \[\hat{A} = \left(\begin{array}{cc}
       A & b \\
       c & d \\
      \end{array}\right)\] is invertible (in $\R^{n+1}$). The affine hyperplane
$H=\{x|\iprod{x}{c}+d=0\}$ is called the defining hyperplane of $F$. Since such
maps are traces of linear maps in $\R^{n+1}$, we sometimes call the image
$F(K)$ of a convex body $K$ under such a map a {\sl projective position of the
body} $K$. These are the only trasformations (on, say, open convex domains),
which map any interval to an interval. Note that affine maps are a subgroup of
fractional linear maps, and that any projective position of a closed ellipsoid
is again an ellipsoid. The same is true for a simplex. For more details,
including proofs of the following useful facts, see \cite{AFM}.

\begin{prop*}
Denote by $H^+$ the half space $\{x_1>1\}$ (where $x_1$ is the first
coordinate of $x$) and let the map $F_0:H^+\to H^+$ (called the {\em
canonical form} of a fractional linear map) be given by \[F_0(x) =
\frac{x}{x_1-1}.\] For any $x_0, y_0 \in \R^n$ and a non-affine
fractional linear map $F$ with $F(x_0) = y_0$, there exist $B, C\in
GL_n$ such that for every $x\in \R^n$, \[ B(F(Cx+x_0)-y_0)=F_0(x).\]
\end{prop*}

Fractional linear maps turn up naturally in convexity. For example,
they are strongly connected to the polarity map, as can be
seen in the following, easily verified proposition.
\begin{prop*}
Let $K\subseteq\{x_1<1\}\subset\R^n$ be a closed convex set
containing $0$. Then for the canonical form $F_0(x)=\frac{x}{x_1-1}$
the following holds: \[F_0(K) = (e_1-K^\circ)^\circ, \] where $e_1=
(1, 0,\ldots, 0)\in \R^n$
\end{prop*}

\begin{rem}\label{Rem_Replace-By-Balls}
Let $K,T\in\K^n$. It is useful to note that when $K\not\subseteq T$, there exist disjoint
closed balls $D_K, D_T$ of positive radius, and a hyperplane $H$ (which
divides $\R^n$ to two open half spaces $H^+$, $H^-$), such that:
$$ D_K\subset K\cap H^-,\qquad T\subset D_T\subset H^+.$$
\end{rem}
This trivial observation means it is sufficient in many cases to
consider the action of fractional linear maps on dilations of the ball.
To this end we write the following easily verified fact, resulting
from a direct computation:
\begin{fact}\label{Fact_FL-of-Balls}
Let $\E_{R,r,\delta}$ stand for the image of the Euclidean ball $D_n$
under the diagonal linear map $A_{R,r}=\text{diag}\{R,r,\dots,r\}$,
shifted by $(1-\delta-R)e_1$, so that the distance of $\E_{R,r,\delta}$
from the hyperplane $H_0=\{x_1=1\}$ is $\delta$. That is: $$\E_{R,r,\delta}=
A_{R,r}D_n+(1-\delta-R)e_1.$$ Then for the canonical form $F_0$ one has:
$$F_0(\E_{R,R,\delta})=\E_{\frac{R}{\delta (\delta+2R)},
\frac{R}{\delta \sqrt{\delta+2R}},\frac{1}{\delta+2R}}.$$
In particular, $F_0(\E_{R,R,\delta})$ contains a translate
of $mD_n$ and is contained in a translate of $MD_n$, where
$m=\frac{R}{\delta}
\min\{\frac{1}{\delta+2R},\frac{1}{\sqrt{\delta+2R}}\}$, and
$M=\frac{R}{\delta}
\max\{\frac{1}{\delta+2R},\frac{1}{\sqrt{\delta+2R}}\}$.
\end{fact}

We shall now prove the main Lemma required for Theorem \ref{Thm_iVol-FL}.
\begin{lem}\label{Lem_FL-Diverg}
Let $K,T\in\K^n$ satisfy $K\not\subseteq T$. For every $\varepsilon>0$
there exist $x_0\in\R^n$ and a fractional linear map $F$ such that:
$$F(T)\subseteq \varepsilon D_n, \qquad D_n+x_0\subseteq F(K).$$
\end{lem}
\begin{proof}
Let $\varepsilon>0$, and let $D_K, D_T$ be balls satisfying the
inclusions in Remark \ref{Rem_Replace-By-Balls}. It suffices to find a
fractional linear map $F$ such that $F(D_K)$ contains a translate of
$D_n$, and $F(D_T)\subseteq \varepsilon D_n$. Without loss of generality
(by applying an affine map), we may assume that the centers of $D_K$
and $D_T$ both lie on the $x_1$ coordinate axis, and that: $$
D_K=\E_{1,1,\delta},\qquad
D_T=\E_{R,R,d},$$ for some $d>2$, and $R,\delta>0$.
From \ref{Fact_FL-of-Balls} it follows that $F_0(D_K)$
contains a translate of $\frac{1}{\delta(\delta+2)}D_n$, and
$F_0(D_T)$ is contained in a translate of $\frac{R}
{d\sqrt{d+2R}}D_n\subset RD_n$. Since $\delta$ is arbitrarily
small, the result follows.
\end{proof}
Theorem \ref{Thm_iVol-FL} now follows in an obvious way.

\section{Comparing convex bodies via Minkowski sums}\label{Sec_Sums}
We will begin with proving the symmetric case.\\
{\bf Proof of Theorem \ref{Thm_Sum-Sym}.} We assume that the symmetric
sets $A, B$ satisfy $|A + K|\le |B+K|$, for all $K\in\K^n$. The case
$n=1$ is trivial. Assume $n\ge 2$ and let $u\in S^{n-1}$. Note that the
inequality $|A + K| \leq |B+K|$ holds also for a compact convex set $K$
with empty interior, and let $K\subset u^\perp$ be an $n-1$ dimensional
Euclidean  ball of ($n-1$ dimensional) volume $1$. The leading coefficient
in $|A+rK|$ for $r\to\infty$ is the width $w_A(u)=2h_A(u)$, and since
$|A+rK|\le|B+rK|$ we have: $$
\forall u\in S^{n-1},\qquad h_A(u)\le h_B(u),$$ as required. \qed

In the case of general bodies, the previous inequality on the widths:
$$\forall u\in S^{n-1},\qquad w_A(u)\le w_B(u),$$ does not imply the
desired inclusion $A\subseteq B$. However, an inequality on the integrals
$\int h_Ad\sigma_K\le\int h_Bd\sigma_K$, against any surface area measure
$\sigma_K$ of a convex body $K$, may be obtained by an argument similar to
that from the proof of Theorem \ref{Thm_Sum-Sym}. It turns out to be
sufficient, by the following theorem from \cite{Lutwak}:
\begin{thm*}[Lutwak] Let $A,B\in\K^n$. Assume that for every simplex
$\Delta\in\K^n$ one has $$V(A,\Delta[n-1])\le V(B,\Delta[n-1]).$$ Then there
exists $x_0\in\R^n$ such that $A+x_0\subseteq B$.
\end{thm*}

\noindent{\bf Proof of Theorem \ref{Thm_Sum-Gen}.} It follows from the
assumption  that for every $K\in\K^n$ and every $\varepsilon>0$, we have
$|K+\varepsilon A|\le |K+\varepsilon B|$. Comparing derivatives at
$\varepsilon=0$ yields an inequality between mixed volumes: $$\forall K\in\K^n,\qquad
V(A,K[n-1])\le V(B,K[n-1]).$$ The conclusion follows from Lutwak's Theorem.\qed

Clearly, these families of functionals are by no means minimal. In the proof
of Theorem \ref{Thm_Sum-Gen} we have in fact used only the functionals
$\{A\mapsto|A+\Delta|\}$. In the symmetric case, we only used dilates of
a flat ball. This leads to the formally stronger formulations of Theorems
\ref{Thm_Sum-Sym}, \ref{Thm_Sum-Gen}.

Let us use these facts to add some information to the well known Busemann-Petty
problem. The Busemann-Petty problem is concerned with comparisons of volume of
central sections. That is, given two centrally symmetric sets $A,B\in\K^n$, satisfying
$|A \cap E| \leq |B \cap E|$ for every $n-1$ dimensional subspace $E$, does it
imply that $|A|\le |B|$? As shown in \cite{GKS, Koldobsky, Zhang}, the answer
is negative for all $n\ge 5$ and positive for $n \leq 4$.
However, if we combine intersections with Minkowski sums, we may use
Theorem \ref{Thm_Sum-Sym} to get:
\begin{cor} \label{Cor_SymmSections}
Let $n\ge2$, and let $A,B\in\K^n$ be centrally symmetric bodies. If for
all $K\in\K^n$ and $E\in G_{n,n-1}$ one has
$$|E\cap(A+K)| \le |E\cap(B+K)|,$$ then $A\subseteq B$. As in
Theorem \ref{Thm_Sum-Sym}, it suffices to fix a body $K_0$, and check
the condition 
only for linear images of $K_0$.
\end{cor}

The non symmetric case is not as simple. In this case we know that for
every $E\in G_{n,n-1}$ there exists a point $x_E\in E$, such that $E\cap
A \subseteq E\cap B+x_E$. However, this does not imply that there exists
a point $x\in\R^n$ such that $A\subseteq B+x$, as shown in the following
example:
\begin{exm}
There exist $A,B\in\K^2$, such that for every line $E$ passing through
the origin, the interval $B\cap E$ is longer than the interval $A\cap E$,
and yet no translation of $B$ contains $A$.

The construction is based on (the dual to) the Reuleaux triangle $R$,
a planar body of constant width: $\forall u\in S^1,\, w_R(u)=
h_R(u)+h_R(-u)=2$. In other words, the projection of $R$ to $u^\perp$
is an interval of length $2$, say $P_{u^\perp}(R)=[\alpha-2,\alpha]$
for some $\alpha=\alpha(u)\in(0,2)$. Then: $$
 |R^\circ\cap u^\perp|= |(P_{u^\perp}(R))^\circ|=
 |    [      1/(\alpha-2),      1 /\alpha     ]|\ge2=
 |D_2\cap u^\perp|.$$ However, no translation of $R^\circ$ contains $D_2$.
\end{exm}

Although Corollary \ref{Cor_SymmSections} may not be extended to non
symmetric bodies, one can show the following fact, in the same spirit:
\begin{thm}\label{thm-NonSymmSections}
Let $n\ge 2$ and $A,B \in \K^n$. If for all $K \in \K^n$ and $E\in G_{n,n-1}$ one has
\begin{equation} \label{eq-Sections}
|E\cap(A+K)|\le |E\cap(B+K)|,
\end{equation}
then $A - A \subset B - B$. In particular, there exist $x_A,x_B\in\R^n$
such that: $$A+x_A\subseteq B-B\subset (n+1)(B+x_B).$$
\end{thm}
\begin{proof}
Apply condition (\ref{eq-Sections}) for convex sets of the form
$-A + K$ to get: $$
|E \cap (A-A + K)| \leq |E \cap (B-A + K)|. $$
Then apply condition (\ref{eq-Sections}) for convex sets of the form
$-B + K$: $$
|E \cap (A-B + K)| \leq |E \cap (B-B + K)|. $$
Since $|E \cap (B-A + K)|=|E \cap (A-B - K)|$, we get for all {\sl
symmetric} $K\in\K^n$: $$
|E \cap (A-A + K)| \leq |E \cap (B-B + K)|, $$
which by Corollary \ref{Cor_SymmSections} implies that $A - A \subset
B-B$. The theorem now follows, since $A-A$ contains a translate of $A$,
and $nB$ contains a translate of $-B$ (as shown by Minkowski. See, e.g., Bonnesen and
Fenchel \cite{Bon-Fen}, $\S7$, 34, pp. 57-58 for a proof).
\end{proof}

We also consider the projection version of Theorem \ref{thm-NonSymmSections}:
\begin{thm} Let $n\ge 2$ and $A, B\in \K^n$. If for all $K\in \K^n$ and
$E\in G_{n,n-1}$ one has
\begin{equation} \label{eq-Projections}
|P_E(A + K)| \leq |P_E(B+K)|,
\end{equation}
then there exists $x_0 \in \R^n$ such that $A+x_0 \subset 2 
B$.
\end{thm}
\begin{proof}
Since projection is a linear operator, $P_E(A + K) = P_E(A) + P_E(K)$.
Due to surjectivity we may rewrite condition (\ref{eq-Projections}) as
follows. For every $n-1$ dimensional subspace $E$ and for every convex
body $K' \subset E$ we have: $$|P_E(A) + K'| \leq |P_E(B) + K'|.$$
Theorem \ref{Thm_Sum-Gen} implies that there exists a shift $x_E \in E$
such that $P_E(A) \subseteq P_E(B) + x_E$. By a result of Chen,
Khovanova, and Klain (see \cite{Klain}) there exists $x_0 \in \R^n$
such that:
$$ A+x_0 \subseteq \frac{n}{n-1}B\subseteq 2B.$$
Note that, while the dimension free constant equals $2$, we have in
fact seen the dimension dependent constant $\frac{n}{n-1}$, which tends
to $1$ when $n\to\infty$.
\end{proof}

Let us formulate a problem which arises from Theorem \ref{thm-NonSymmSections}: \\
{\bf Problem A:} Let $A, B\in\K^n$ such that $0 \in int(A\cap B)$.
Assume that for every $n-1$ dimensional subspace $E \in G_{n,n-1}$ there
exists $x_E \in \R^n$ such that $A \cap E + x_E \subset B \cap E$. Does
there exist a universal constant $C > 0$ such that $A + x_0 \subset C B$
for some $x_0 \in \R^n$? Of course, we are interested in $C$ independent
of the dimension. We suspect that $C = 4$ suffices.

\section{Comparing n-tuples - Affine case}\label{Sec_ntuples-AFF}
Although very little may be said about inclusion of the convex body in
another one using only affine images of those bodies, some information is
anyway available through the use of mixed volumes.

\begin{thm} \label{thm-Identify1}
Let $A, B, K_2, \ldots, K_n\in\K^n$ be centrally symmetric bodies. If
for all $u \in SL_n$ one has
\begin{equation} \label{eq-Identify1}
V(uA, K_2, \ldots, K_n) \leq V(uB, K_2, \ldots, K_n),
\end{equation} then $A \subseteq B$.
\end{thm}
\begin{proof}
Due to multilinearity of mixed volume, (\ref{eq-Identify1}) holds
for any $u \in GL_n$ as well. Since $V(uK_1, \ldots, uK_n)= |\det(u)|
V(K_1, \ldots, K_n)$, we get for every $u \in GL_n$: $$
V(A, uK_2, \ldots, uK_n) \leq V(B, uK_2, \ldots, uK_n).$$
Fix a direction $v\in S^{n-1}$, and let $E=v^\perp$. We may choose a
sequence $\{u_n\} \subset GL_n$ such that for all $K\in\K^n$, $u_n(K)
\rightarrow P_E(K)$ in the Hausdorff metric. By continuity of mixed
volume with respect to the Hausdorff metric we get that: $$
V(A, P_E K_2, \ldots, P_E K_n) \leq V(B, P_E K_2, \ldots, P_E K_n),\,
\mbox{that is:}$$ 
$$w_A(v)\cdot V(P_E K_2, \ldots, P_E K_n) \le
  w_B(v)\cdot V(P_E K_2, \ldots, P_E K_n). $$
In the last inequality, $V$ stands for the $n-1$ dimensional mixed
volume, and since $K_i$ have non empty interior, it does not vanish.
Thus for all $v\in S^{n-1}$, $w_A(v) \leq w_B(v)$,  which implies
inclusion, since the bodies are symmetric.
\end{proof}

\begin{thm}\label{thm-tuplesLinearSeparation}
Let $A_1,B_1,\dots,A_n,B_n\in\K^n$ be centrally symmetric bodies.
If for all $u_1,\dots,u_n \in SL_n$ one has
\begin{equation}\label{eq-Identify2}
V(u_1A_1, \dots, u_n A_n) \le V(u_1 B_1, \dots, u_n B_n),
\end{equation}
then there exist positive constants $t_1, \dots, t_n$ such that
$\Pi_1^n t_i = 1$ and $A_i \subset t_i B_i$.
\end{thm}
\begin{proof} For $2\le i\le n$, denote: $$
t_i = \max_{t > 0}\{t : \quad tA_i \subseteq B_i\},$$
and let $v_i \in S^{n-1}$ be such that $t_i h_{A_i}(v_i) = h_{B_i}(v_i)$.
Since we may rotate the bodies $A_i$ separately as desired, we may
assume without loss of generality that $v_i=e_i$. Let $v\in S^{n-1}$,
let $g$ be a rotation such that $g(v)=e_1$, and denote $t_1=
(t_2 \dots t_n)^{-1}$. Due to continuity of mixed volume,
(\ref{eq-Identify2}) holds also for degenerate $u_i$. Applying it to
$g\circ P_v,P_{v_2},\dots, P_{v_n}$ (where $P_w$ denotes the orthogonal
projection onto the $1$-dimensional subspace spanned by $w$) yields:
\begin{eqnarray*}
2^n h_{A_1}(v) \cdot h_{A_2}(v_2) \cdot \ldots \cdot h_{A_n}(v_n)
&=&
V(g P_v A_1,P_{v_2} A_2, \dots, P_{v_n} A_n)
 \\ \\ &\le&
V(g P_v B_1,P_{v_2} B_2, \dots, P_{v_n} B_n)
\\ \\ &=&
2^n h_{B_1}(v) \cdot h_{B_2}(v_2) \cdot \ldots \cdot h_{B_n}(v_n)
\\ \\ &=&
\frac{2^n}{t_1} h_{B_1}(v) \cdot h_{A_2}(v_2) \cdot \ldots \cdot h_{A_n}(v_n).
\end{eqnarray*}
This implies $h_{t_1A_1}(v)\le h_{B_1}(v)$. Since $v$ was arbitrary,
$t_1A_1 \subseteq B_1$. This completes the proof, since for $i\ge 2$,
$t_iA_i\subseteq B_i$ by the definition of $t_i$, and
$\Pi_{i=1}^n t_i=1$. \qed

\begin{cor} \label{cor-2}
Let $A, B\in\K^n$ be centrally symmetric bodies. If for all $u_i\in
SL_n$ one has $$
V(u_1A, \ldots, u_nA) \le V(u_1B, \ldots, u_n B),$$
then $A \subseteq B$.
\end{cor}
\end{proof}

\begin{rem}
Theorem \ref{thm-tuplesLinearSeparation} is not true if the bodies are
not assumed to be centrally symmetric. One example in $\R^2$ is given by
the bodies $K_1=R$, the Reuleaux triangle, and $K_2=L_1=L_2=D_2$, the
Euclidean unit ball. In fact for any $n\ge2$, the choice $K_1=-L_1=A$ and
$L_i=K_i=S_i$ yields a counter example, for trivial reasons, provided that
$A\not=-A$, and $S_i=-S_i$, for $i\ge 2$.
\end{rem}

Let us mention a problem inspired by Corollary \ref{cor-2}. \\
{\bf Problem B:} Let $K, L$ be convex bodies with barycenter at the origin.  Denote
\[
a(K, L) = \sup_{x\neq 0} \frac{h_K(x)}{h_L(x)}.
\]
Assume that for any $n-1$ dimensional subspace $E$ we have
\[
a(K,L)|P_E(K)| \leq |P_E(L)|.
\]
Does this imply that $|K| \leq |L|$?

\section{Comparing $n$-tuples - Fractional linear case}\label{Sec_ntuples-FL}
Let us denote by $K_{(0)}^n$ the subset of $\K^n$ of bodies with $0$ in their interior. We have:
\begin{thm}
Let $n\ge1$ and let $K_1,\dots,K_n,L_1,\dots,L_n\in\K_{(0)}^n$ such
that for every $\lambda_1,\dots,\lambda_n>0$ and every admissible (for
$\lambda_1 K_1, \dots, \lambda_n K_n$) $F\in FL_n$ one has
$$V(F(\lambda_1K_1),\dots,F(\lambda_nK_n))\le
V(F(\lambda_1L_1),\dots,F(\lambda_nL_n)).$$
Then for every $1\le i\le n$, $K_i\subseteq L_i$.
\end{thm}

\begin{proof} Assume the claim is false, for example $K_1 \not
\subseteq L_1$. Apply Remark \ref{Rem_Replace-By-Balls} for $K_1, L_1$
and select $H^+$ to contain $0$. Let $D_{K_1},D_{L_1}$ be the balls from
\ref{Rem_Replace-By-Balls} (note again that these balls are not
necessarily centered at $0$), and take a ball $D$ centered at the origin
such that $D \subset D_{L_1}$. Without loss of generality (by a correct
rescaling) we may assume that $D$ is the Euclidean unit ball $D_n$. Fix
$\lambda_2, \ldots, \lambda_n > 0$ such that $\lambda_i K_i, \lambda_i L_i$
are contained in the unit ball, and choose $\varepsilon_0$ such that
for every $2\le i\le n$: $$
\varepsilon_0 D_n \subseteq \lambda_i K_i\subseteq D_n,\qquad
\varepsilon_0 D_n \subseteq \lambda_i L_i\subseteq D_n.
$$
It follows that for every admissible $F$ we have:
\begin{equation}\label{Eq_Needed-to-Ref}
V(F(D_{K_1}),F(\varepsilon_0 D_n)[n-1])\le
V(F(D_{L_1})[n]).
\end{equation}
Without loss of generality, we may assume that the centers of the balls
$D_{K_1},\,\varepsilon_0D_n$, and $D_{L_1}$ are collinear (for example,
we may replace $D_{L_1}$ by a larger ball containing it, while keeping
the intersection $D_{L_1}\cap D_{K_1}$ empty). By applying an affine map
$A$ we may further assume that:$$
A(D_{K_1})           =\E_{1,1,\delta},\quad
A(D_{L_1})           =\E_{R,R,d},\quad
A(\varepsilon_0D_n)=\E_{\varepsilon_1,\varepsilon_1,d_1},$$
where $R,\varepsilon_1>0$, $d,d_1>2$, and $\delta>0$ is arbitrarily
small. By \ref{Fact_FL-of-Balls}, we have:
$$\frac{1}{\delta(\delta+2)}D_n\subset F_0AD_{K_1},\quad
\frac{\varepsilon_1}{d_1(d_1+2\varepsilon_1)}\subset
F_0A\varepsilon_0D_n,\quad
F_0AD_L\subset\frac{R}{d\sqrt{d+2R}}D_n.$$
Substituting $F=F_0A$ in (\ref{Eq_Needed-to-Ref}), and using
multilinearity of mixed volumes, we get:
$$
\frac{1}{\delta(\delta+2)} 
\left(\frac{\varepsilon_1}{d_1(d_1+2\varepsilon_1)}\right)^{n-1}\le
\left(\frac{R}{d\sqrt{d+2R}}\right)^n,$$
which is false for sufficiently small $\delta>0$. Thus
$K_1\subseteq L_1$, as required.
\end{proof}

\end{document}